\documentclass[11pt,letterpaper]{amsart}
\usepackage{epsfig,layout,graphics}
\usepackage{pgf,pgfarrows,pgfnodes,pgfautomata,pgfheaps}
\usepackage{epstopdf}
\usepackage{multicol}
\usepackage{amsmath,amssymb}
\usepackage{mathtools}
\usepackage{multicol}

 \theoremstyle{plain}
 \newtheorem{thm}{Theorem}[section]
 \numberwithin{equation}{section} 
 \numberwithin{figure}{section} 
 \theoremstyle{plain}
 \newtheorem{prop}[thm]{Proposition}
 \theoremstyle{plain}
 \theoremstyle{plain}
 \newtheorem{theorem}[thm]{Theorem}
 \theoremstyle{plain}
 \newtheorem{corollary}[thm]{Corollary}
\theoremstyle{plain}
 \newtheorem{remark}[thm]{Remark}
 \theoremstyle{plain}
 \newtheorem{lemma}[thm]{Lemma}
\def\A{{\mathcal{A}}}
\def\B{{\mathcal {B}}}

\def\L{{\mathcal{L}}}

\def\C{{\mathcal{C}}}

\def\M{{\mathcal{M}}}

\def\Forall{\quad \hbox{ for all }}

\def\<{{\langle}}
\def\>{{\rangle}}
\def\R{\mathbb{R}}

\def\bu{{\bf u}}
\def\bv{{\bf v}}

\def\bv{{\mathbf v}}
\def\bu{{\mathbf u}}

\def\bV{{\mathbf V}}

\def\du#1#2#3{\overset{#3}{\underset{#2}{#1}}}

\begin{document}
\bibliographystyle{plain}  
\title[Notes on Symmetric SPP]
{A note on  stability and optimal approximation estimates \\ for symmetric saddle point systems}

\author{Constantin Bacuta}
\address{University of Delaware,
Department of Mathematics,
501 Ewing Hall 19716}
\email{bacuta@math.udel.edu}

\keywords{saddle point system, stability estimates, spectral estimates, golden ratio}
\subjclass[2000]{74S05, 74B05, 65N22, 65N55}

\begin{abstract} 

We  establish sharp  well-posedness and approximation estimates  for variational  saddle point systems at the continuous level. The main results of this note  have been  known to be true only in the finite dimensional case. Known spectral results from the discrete case are reformulated and proved using a functional analysis view, making the proofs in both cases, discrete and continuous, less technical than the known discrete approaches.  
We focus on analyzing the special case when the  form  $a(\cdot, \cdot)$  is  bounded, symmetric, and  coercive, and the mixed form $b(\cdot, \cdot)$  is bounded and satisfies a standard $\inf-\sup$  or LBB   condition. We characterize the spectrum of the symmetric  operators that describe the problem at the continuous level.  For a particular choice of the inner product on the product space of $b(\cdot, \cdot)$, we prove that the spectrum of the  operator representing the system at continuous level is $\left \{\frac{1-\sqrt{5}}{2}, 1,  \frac{1+\sqrt{5}}{2} \right \}$.  As consequences of the spectral description, we find the minimal length interval that contains the ratio between the norm of the data and the norm of the solution, and  prove explicit approximation estimates that  depend  only on  the continuity constant and the continuous and the discrete  $\inf-\sup$ condition constants. 
\end{abstract}

\maketitle
\section{Notation and standard properties}\label{section: notation}

 The  existing literature on stability and approximation estimates for symmetric Saddle Point (SP) systems is quite rich for 
both continuous and discrete levels. While at discrete level the estimates can be done using eigenvalue analysis of symmetric matrices and consequently are optimal, at the continuous level, the estimates are presented as inequalities depending on related constants 
 and consequently are not optimal.  In this note, we will  establish optimal estimates at the continuous level, that can be viewed as generalizations of results at the discrete level. The new spectral estimates provide more insight into the behavior of the general symmetric SP problems. In addition, information about the continuous  spectrum and the techniques used to characterize it can lead to efficient  analysis of  iterative methods for SP systems. 
 
 Towards this end, we let  $\bV$ and  $Q$ be two Hilbert spaces with  inner products given by {\it symmetric} bilinear forms $a(\cdot, \cdot)$ and $(\cdot, \cdot)$ respectively, with the corresponding induced norms $|\cdot|_{\bV} =|\cdot| =a(\cdot, \cdot)^{1/2}$ 
and $\|\cdot\|_Q=\|\cdot\| =(\cdot, \cdot)^{1/2}$. The dual pairings on $\bV^* \times \bV$ and 
$Q^* \times Q$ are denoted by $\langle \cdot, \cdot \rangle$. Here, 
$\bV^* $ and $Q^*$ denote the duals of $\bV$ and $Q$, respectively.  
With the inner products  $a(\cdot, \cdot)$ and $(\cdot, \cdot)$, we associate the operators \\
$\A:V\to V^*$ and $\C:Q\to Q^*$ defined by 
\[
\langle \A \bu,\bv \rangle=a(\bu, \bv)  \Forall \bu ,\bv \in \bV 
\]
and
\[
\langle \C p,q\rangle=(p, q) \Forall  \ p, q \in Q. 
\]
The operators $\A^{-1}:V^*\to V$ and $\C^{-1}:Q^*\to Q$ are  the Riesz-canonical representation operators and satisfy
\begin{equation}\label{eqn:A-1}
a(\A^{-1}\bu^*, \bv)= \<\bu^*, \bv\>, \ \  |\A^{-1} \bu^*|_{\bV} = \|\bu^*\|_{\bV^*},    \bu^*  \in \bV^*,  \bv \in \bV,
\end{equation} 
\begin{equation}\label{eqn:C-1}
(\C^{-1} p^*, q) = \<p^*, q\>, \ \ \| \C^{-1} p^* \| =\|p^*\|_{Q^*},   \ p^* \in Q^*, q \in Q. 
\end{equation}

Next, we suppose that $ b(\cdot, \cdot) $ is a continuous 
bilinear form on $\bV \times Q$, satisfying the inf-sup condition. More precisely, we assume that
 \begin{equation} \label{inf-sup_a}
 \du{\sup}{p \in Q}{} \ \du {\sup} {\bv \in \bV}{} \ \frac {b(\bv, p)}{\|p\| \ |\bv|} =M <\infty, \ \text{and} \ \ 
     \du{\inf}{p \in Q}{} \ \du {\sup} {\bv \in \bV}{} \ \frac {b(\bv, p)}{\|p\| \ 
|\bv|} =m>0.
\end{equation} 
Throughout this paper, the ``inf"  and ``sup" are  taken  over nonzero vectors. 
With the form $b$, we associate the linear operators 
$B:V\to Q^*$ and \\  $B^*:Q\to V^*$ defined by
\[
\<B \bv,q\>=b(\bv, q)= \langle B^*q, \bv \rangle  \Forall \bv \in \bV, \ q \in Q.
\]
Let $\bV_0=\ker(B) \subset \bV$, be  the kernel of $B$. 

 For a bounded linear operator $T:X \to Y$ between two Hilbert spaces $X$ and $Y$, we denote the Hilbert transpose of $T$ by $T^t$. If $X=Y$, we say that $T$ is symmetric if $T=T^t$. For a bounded  linear  operator $T:X\to X$ , we denote  the  spectrum of the operator $T$ by $\sigma(T)$.

Next, we review   the Schur complement operator with the notation introduced in \cite{B09}. First, we notice that the operators $\C^{-1}B:\bV \to Q$  and $\A^{-1} B^*:Q\to \bV$ are  symmetric  to each other, i.e., 
\begin{equation}\label{eq:sym}
(\C^{-1}B \bv, q) =a(\bv, \A^{-1}B^*q), \ \bv \in \bV, q \in Q.
\end{equation}
Consequently, 
 $
 (\C^{-1} B)^t = \A^{-1} B^*$  and $(\A^{-1} B^*)^t= \C^{-1} B$. The  Schur complement  on $Q$  is the operator 
 $S_0:=\C^{-1}B \A^{-1} B^*: Q \to Q$. The   operator $S_0$ is {symmetric and  positive definite  on $Q$}, satisfying 
\begin{equation}\label{eq:sigmaS0}
\sigma(S_0) \subset [m^2, M^2],  \ \text{and}  \ m^2, M^2 \in  \sigma(S_0). 
\end{equation}
A proof of \eqref{eq:sigmaS0} can be found in \cite{B06, B09}.
Consequently,  for any $ \ p \in Q$,  \begin{equation}\label{A-1Biso}
M\|p\| \geq   \|p\|_{S_0}:=(S_0 p, p)^{1/2} = |\A^{-1} B^* p|_{\bV} \geq m \|p\|. 
\end{equation}

For $f\in \bV^*$, $g\in Q^*$, we consider the following variational problem: \\
 Find $(\bu, p)  \in \bV \times Q$ such that 
\begin{equation}\label{abstract:variational}
\begin{array}{lcll}
a(\bu,\bv) & + & b( \bv, p) &= \langle {\bf f},\bv \rangle \ \Forall  \bv \in \bV,\\
b(\bu,q) & & & =\<g,q\>  \  \Forall  q \in Q, 
\end{array}
\end{equation}
where  the  bilinear form  $b:\bV \times Q \to \R$ satisfies  
\eqref{inf-sup_a}. It is known that the above variational problem  has unique solution for any $f\in \bV^*$, $g\in Q^*$ (see some of the original proofs in \cite{A-B, babuska73, brezzi74}). Further  results on stability and well poseddness of the problem, can be found in many publications, e.g., 
\cite{braess, brenner-scott, falk-osborn80, arnold-brezzi-fortin84, girault-raviart, brezzi-fortin,  ern-guermond, monk91,  quarteroni-valli-94, brezzi2003, elman-silvester-wathen,  MardalWinther2011, B09}.  


The  operator version of 
the problem  \eqref{abstract:variational} is: \\
Find $(\bu, p)  \in \bV \times Q$ such that
\begin{equation}\label{abstract:operational}
\begin{array}{lcll}
\A \bu  & + & B^* p &= {\bf f},\\
B \bu  & & & =g. 
\end{array}
\end{equation}
By applying   the Riesz representation operators $\A^{-1} $ and $\C^{-1} $  to the first and the second  equation respectively, we obtain   the system
\begin{equation}\label{eq:SP-Hilbert}
 \begin{pmatrix} 
    I   &    \A^{-1} B^* \\
       \C^{-1} B &  0  
   \end{pmatrix}   
    \begin{pmatrix} 
  \bu  \\
     p   
   \end{pmatrix}   =
     \begin{pmatrix} 
 \A^{-1} {\bf f}  \\
     \C^{-1}g   
   \end{pmatrix}.
\end{equation}

Since $S_0$ is an invertible operator on $Q$, \eqref{eq:SP-Hilbert} is also equivalent to 
\begin{equation}\label{eq:SP-HilbertS}
 \begin{pmatrix} 
    I   &    \A^{-1} B^* \\
      S_0^{-1} \C^{-1} B &  0  
   \end{pmatrix}   
    \begin{pmatrix}
  \bu  \\
     p   
   \end{pmatrix}   =
     \begin{pmatrix} 
 \A^{-1} {\bf f}  \\
    S_0^{-1} \C^{-1}g  
   \end{pmatrix}.
\end{equation}   
The matrix operators associated with \eqref{eq:SP-Hilbert} and \eqref{eq:SP-HilbertS} will be investigated in  Section \ref{section:stability} and in Section \ref{section:approximationSPP}. 
\section{Spectral and stability estimates} \label{section:stability}
In this section, we establish sharp stability estimates for the problem \eqref{abstract:variational}. 
If $T$ is a bounded  invertible operator  on a Hilbert space $X$, and $x\in X$ is the uniques solution of $Tx=y$, then 
 \begin{equation}\label{eq:Tstable} 
\|T\|^{-1} \|y\|_X \leq \|x\|_X \leq \|T^{-1}\| \|y\|_X, 
\end{equation}
where $\|\cdot \|_X$ is the Hilbert norm induced by the inner product on $X$, and $\|T\|$ is the standard norm on $\L(X,X)$ induced by $\|\cdot \|_X$. 
If in addition, we have that $T$ is symmetric, then $\sigma(T) \subset \R$ and 
 \begin{equation}\label{eq:normT} 
\|T\| = \sup \left \{|\lambda| : \lambda \in \sigma(T)\right \}, \ \text{and} \ \ 
  \|T^{-1}\| = \sup \left \{\frac{1}{|\lambda|} :  \lambda \in \sigma(T) \right \}.
\end{equation}
 
We note that the operator  $T:\bV \times Q \to \bV \times Q $, \  $T: = \begin{pmatrix} 
    I   &    \A^{-1} B^* \\
      \C^{-1} B &  0  
   \end{pmatrix}$  
    is symmetric with respect to the inner product    
   \begin{equation} \label{eq:natiralIP}
 \left  (
   \begin{pmatrix}
  \bu  \\
     p   
   \end{pmatrix} , 
   \begin{pmatrix}
  \bv  \\
     q   
   \end{pmatrix}
 \right )_{\bV \times Q}:=a(\bu, \bv) + (p, q), 
   \end{equation}
and the operator $T_{S_0}:\bV \times Q \to \bV \times Q $, \    $T_{S_0}: = \begin{pmatrix} 
    I   &    \A^{-1} B^* \\
      S_0^{-1} \C^{-1} B &  0  
   \end{pmatrix}$
   is symmetric with respect to the $S_0$-weighted  on $Q$  inner product  
   \[
 \left  (
   \begin{pmatrix}
  \bu  \\
     p   
   \end{pmatrix} , 
   \begin{pmatrix}
  \bv  \\
     q   
   \end{pmatrix}
 \right )_{\bV \times Q_{S_0}}:=a(\bu, \bv) + (S_0p, q)=a(\bu, \bv) + (p, q)_{S_0}.  
   \]
   Thus, $\sigma(T)$ and $\sigma(T_{S_0})$ are compact subsets of $\R$. Due to the close relation of the two operators with $S_0$ we will establish estimates for  $\sigma(T)$ and find  $\sigma(T_{S_0})$.  First, we introduce the following  numerical values  
 \[
 \lambda^{\pm}_m:= \frac{1 \pm \sqrt{4m^2 +1}}{2}, \ \text{and} \ \  \lambda^{\pm}_M:= \frac{1 \pm \sqrt{4M^2 +1}}{2}.
 \]  
 \begin{lemma} \label{lemma:1} Assume that $\bV_0=\ker(B)$ is non-trivial.  Then, the spectrum of $ T_{S_0} $ is discrete and 
 \begin{equation}\label{eq:SpT_S} 
 \sigma(T_{S_0})=\left \{\frac{1-\sqrt{5}}{2}, 1,  \frac{1+\sqrt{5}}{2} \right \}.  
 \end{equation}
The spectrum of $T$  satisfies the following inclusion properties
 \begin{equation}\label{eq:SpT} 
 \{\lambda^{\pm}_m, 1, \lambda^{\pm}_M\} \subset  \sigma(T)\subset  [\lambda_M^{-},  \lambda_m^{-}] \cup \{ 1\} \cup [\lambda_m^{+}, \lambda_M^{+}]. 
 \end{equation}
 \end{lemma}  

 To give a more  fluid presentation of the main results, we postpone the proof of  the Lemma \ref{lemma:1} for Section \ref{section:appendix}. 
 \begin{remark} In the discrete case, Lemma \ref{lemma:1}  is known, in the context of {\it block diagonal prconditioning} of  saddle point systems, at least since the work of Kuznetsov  in  \cite{kuznetsov95}, or  the works of  Silvester and  Wathen, in \cite{silvester-wathen}  and Murphy, Golub and Wathen in \cite{murphy-golub-wathen00}. Other related estimates can be found  in  in  Section 10.1.1 of the review paper of  Benzi, Golub, and Liesen  \cite{benzi-golub-liesen} and the refereces therein. 
 \end{remark}
\begin{theorem}\label{th:T1}  If  $ (\bu, p) \neq (0, 0) $  is  the solution of  \eqref{abstract:variational} then 
 \begin{equation}\label{eq:normu-pS} 
 \frac{ \left ( \|{\bf f}\|^2_{\bV^*} + \|\C^{-1} g\|^2_{S_0^{-1}} \right )^{1/2}}{\left ( |\bu|^2  + \|p\|^2_{S_0} \right )^{1/2}}
\in \left [\frac{\sqrt{5}-1}{2}, \frac{\sqrt{5}+1}{2} \right ],
 \end{equation}
and 
 \begin{equation}\label{eq:normu-p} 
\frac{ \left ( \|{\bf f}\|^2_{\bV^*} + \|g\|^2_{Q^*} \right )^{1/2}}{\left ( |\bu|^2  + \|p\|^2\right )^{1/2}}
\in [\, |\lambda_m^{-}|,|\lambda_M^{+}|\,].
 \end{equation}
\end{theorem} 
 
\begin{proof}
Let   $(\bu, p) \in \bV \times Q$ be the solution of  \eqref{abstract:variational}. Then, using \eqref{eq:SP-HilbertS},   $\begin{pmatrix} \bu \\ p  \end{pmatrix} = T_{S_0} ^{-1}    \begin{pmatrix} 
 \A^{-1} {\bf f}  \\
    S_0^{-1} \C^{-1}g  
   \end{pmatrix}$.  
 The estimate \eqref{eq:normu-pS} is a direct consequence of  \eqref{eq:Tstable}, \eqref{eq:normT}, Lemma 
 \ref{lemma:1}, and the fact that 
 \[
 \left \| \begin{pmatrix} 
 \A^{-1} {\bf f}  \\
    S_0^{-1} \C^{-1}g  
   \end{pmatrix} \right \|^2_{\bV \times Q_{S_0}} =  \|{\bf f}\|^2_{\bV^*} + (S_0^{-1} \C^{-1}g, S_0^{-1} \C^{-1}g)_{S_0} =
   \|{\bf f}\|^2_{\bV^*} + \|\C^{-1} g\|^2_{S_0^{-1}}.
 \]
 For the the proof of \eqref{eq:normu-p} we use that 
 $ \begin{pmatrix} \bu \\ p  \end{pmatrix} = T^{-1}    \begin{pmatrix} 
 \A^{-1} {\bf f}  \\
 \C^{-1}g  
   \end{pmatrix}$. 
\end{proof}
As a direct consequence of Theorem \ref{th:T1} we have:
 \begin{corollary}  If  $(\bu, p) \in \bV \times Q$  is the solution of  \eqref{abstract:variational} then
 \begin{equation}\label{eq:normu-pSRHS}
\left ( |\bu|^2  + \|p\|^2_{S_0} \right )^{1/2}
\leq    \frac{2}{\sqrt{5}-1}    \left ( \|{\bf f}\|^2_{\bV^*} + \frac {1} {m^2}\|\ g\|^2_{Q^*}\right )^{1/2}, 
 \end{equation}
 
  \begin{equation}\label{eq:normu-pSLHS}
\left ( |\bu|^2  + \|p\|^2_{S_0} \right )^{1/2}
\geq \frac{2}{\sqrt{5}+1} \left ( \|{\bf f}\|^2_{\bV^*} + \frac {1} {M^2}\|\ g\|^2_{Q^*} \right )^{1/2},  
 \end{equation}

 \begin{equation}\label{eq:normu-pRHS}
\left ( |\bu|^2  + \|p\|^2 \right )^{1/2}
\leq    \frac{2}{\sqrt{4m^2 +1}-1}    \left ( \|{\bf f}\|^2_{\bV^*} +\|g\|^2_{Q^*} \right )^{1/2}, 
 \end{equation}
 
  \begin{equation}\label{eq:normu-pLHS}
\left ( |\bu|^2  + \|p\|^2 \right )^{1/2}
\geq \frac{2}{\sqrt{M^2+1}+1} \left ( \|{\bf f}\|^2_{\bV^*} + \|g\|^2_{Q^*} \right )^{1/2}.
 \end{equation}
\end{corollary}


\section{The Xu-Zikatanov approach for the symmetric case} \label{section:approximation} 
In this section we present the  Xu and Zikatanov  estimate  for the  Galerkin approximation of variational problems in 
the {\it  symmetric} case and provide a spectral description  of  the estimating constants. 

Let $X$ be a Hilbert space, let $\B: X \times X \to \R$ be  a {\it symmetric} bilinear form satisying
 \begin{equation} \label{inf-sup_B}
M_{\B}:= \du{\sup}{x \in X}{} \ \du {\sup} {y  \in X}{} \ \frac {\B(x, y)}{\|x\|_X \ 
\|y\|_X}<\infty  \ \ \text{and} \  \  m_{\B}:=\du{\inf}{x \in X}{} \ \du {\sup} {y  \in X}{} \ \frac {\B(x, y)}{\|x\|_X \ 
\|y\|_X}>0. 
\end{equation} 
For any $F \in X^*$, we consider the problem: \\
Find $x \in X$ such that
\begin{equation}\label{eq:Bproblem}
\B(x,y) =\<F,y\>, \Forall \ y \in X.
\end{equation}

Let $T:X\to X$ be the {\it symmetric} operator associated to the form $\B(\cdot, \cdot)$,
\begin{equation}\label{eq:Tdef}
\< Tx, z\>_{X}=\B(x, z), \Forall \ x, z \in X.
\end{equation}

Next result gives a characterization for the  invertibility of a symmetric bounded  operator $T:X \to X $ with   $X$ a Hilbert space   
and is a direct consequence of  the {\it bounded inverse theorem}, see e.g. Theorem 3.8 in \cite{schechter}, and of the fact that for an {\it injective  and symmetric}  bounded  operator \\ $T:X \to X $, we have $range(T^t) =range(T) = X$. 

\begin{prop} \label{prop:p2} 
A  bounded symmetric  operator $T$   on a Hilbert space $X$  is invertible if and only if 
there exists $\delta >0$ such that $\|Tx \| \geq \delta \|x\|$ for all  $x \in X$. 
 \end{prop} 

From  the first part of \eqref{inf-sup_B}, we obtain that  $T$ is  bounded and $\|\B\|:=\|T\|=M_{\B}$. 
The second part of assumption \eqref{inf-sup_B} implies that $\|T x\|_X \geq m_{\B} \|x\|_X $, for all  $x \in X$. Since $T$ is a symmetric operator, by Proposition \ref{prop:p2}, we get that  $T$ is invertible. Consequently, the problem  \eqref{eq:Bproblem}  has unique solution for 
 any $F \in X^*$. We note  that  the second part of   \eqref{inf-sup_B}  is equivalent to 
 $
 \|T^{-1}\| = \frac{1}{m_{\B}}. 
 $
Thus,  from \eqref{eq:Tstable}, using the Riesz representation theorem, we get that the solution of  \eqref{eq:Bproblem} satisfies
\begin{equation}\label{eq:Tstable2} 
1/ M_\B  \ \|F\|_{X^*} \leq \|x\|_X \leq  1/ {m_{\B}}\  \|F\|_{X^*}. 
\end{equation}
Next, we let $X_h\subset X$, be a finite dimensional approximation space, and consider the following discrete variational problem: \\
Find $x_h \in X_h$ such that
\begin{equation}\label{eq:Bproblem_h}
\B(x_h,y) =\<F,y\>, \Forall \ y \in X_h.
\end{equation}
Let $T_h:X_h\to X_h$ be the {\it symmetric} operator defined by 
\[
\< T_h x_h, z_h\>_{X}=\B(x_h, z_h), \Forall \ x_h, z_h \in X_h.
\]

Assuming the  discrete $\inf-\sup$ condition 
 \begin{equation} \label{inf-sup_B-h}
\du{\inf}{x_h \in X_h}{} \ \du {\sup} {y_h  \in X_h}{} \ \frac {\B(x_h, y_h)}{\|x_h\|_X \ 
\|y_h\|_X} :=m_{\B_h}>0,
\end{equation} 
the discrete problem \eqref{eq:Bproblem_h} has unique solution  $x_h$, called the Galerkin approximation of the continuous solution $x$ of  \eqref{eq:Bproblem}. The following result improves on  the estimate of Aziz and Babu\v{s}ka in \cite{A-B}, and was proved in a more general case by Xu and Zikatanov   in \cite{xu-zikatanov-BBtheory}. 
\begin{theorem}\label{th:Xu-Zikatanov}
Let $x$ the solution of \eqref{eq:Bproblem} and let  $x_h$ be  the solution of  \eqref{eq:Bproblem_h}. Under the assumptions 
\eqref{inf-sup_B} and \eqref{inf-sup_B-h}, the following error estimate holds
 \begin{equation} \label{eq:Error-h}
\| x -x_h\|_X \leq \frac{M_\B}{m_{\B_h}}\  \du{\inf}{y_h \in X_h}{} \ \|x-y_h\|_X.  
 \end{equation} 
\end{theorem}
The estimate is based on the fact that the operator $\Pi: X \to X_h$ defined by $\Pi x= x_h$, where
\[
\B(x_h, z_h)=\B(x, z_h) , \Forall  z_h \in X_h,
\]
is a projection, and $\|\Pi\|_{\L(X,X)} = \|I- \Pi\|_{\L(X,X)}$, see \cite{kato, xu-zikatanov-BBtheory}.  Since our operators $T$ and $T_h$ are symmetric,  we can characterize the constants $M_\B$ and $m_{\B_h} $ using \eqref{eq:normT}. Thus, we have 
 \[
M_\B = \|T\|\  =\max  \{|\lambda| : \lambda \in \sigma(T)\}, \ \ \text{and} 
\]
 \[ m_{\B_h}  =  {1}/{ \|T_h^{-1}\|}  =
\min  \{|\lambda| : \lambda \in \sigma(T_h)\}. 
\] 

 \section{Approximation estimates for the coercive symmetric case} \label{section:approximationSPP} 
 In this section, we apply the approximation result of Theorem \ref{th:Xu-Zikatanov} to the  discrete approximation   of problem \eqref{abstract:variational}.  In what follows, we work with  the setting and   the notation introduced  in Section \ref{section: notation}.  
  Let $\bV_h$ be a subset of $ \bV$ and let $ {\M}_h $ be a finite dimensional  subspace  of $ Q$. 
 We consider the  restrictions of the forms   $a(\cdot, \cdot)$ and $b(\cdot, \cdot)$ to the   discrete spaces  $\bV_h$ and $M_h$ and define the corresponding discrete  operators $A_{h}, \C_{h}, B_h$, and $B_h^*$.  For example, $A_{h} $ is the discrete version of $A$, and    is defined by   
 \[
\<A_{h}  \bu_h , \bv_h\>=a( \bu_h , \bv_h), \Forall \bu_h \in \bV_h, \bv_h \in \bV_h. 
\] 
We assume that  there exists  $m_h>0$ such that 
 \begin{equation} \label{inf-sup_h}
\du{\inf}{p_h \in M_h}{} \ \du {\sup} {\bv_h \in \bV_h}{} \ \frac {b(\bv_h, p_h)}{\|p_h\|\ 
|\bv_h|} =m_h>0.
\end{equation} 

We define the discrete Schur complement   $S_{0,h}: \M_h \to \M_h$ by $S_{0,h}:=\C^{-1}_{h} B_h A_h^{-1} B_h^*$, and notice that, 
$m_h^2$ is  the smallest  eigenvalues of  $\sigma( S_{0,h})$.

We consider the discrete variational form of \eqref{abstract:variational}: \\  Find $(\bu_h, p_h)  \in \bV_h \times \M_h$ such that
\begin{equation}\label{abstract:variational-h}
\begin{array}{lcll}
a(\bu_h,\bv_h) & + & b( \bv_h, p_h) &= \langle {\bf f},\bv_h \rangle \ \Forall  \bv_h \in \bV_h,\\
b(\bu_h, q_h) & & & =\<g,q_h\>  \  \Forall  q_h \in \M_h, 
\end{array}
\end{equation}
 It is  well known, see e.g.,   \cite{braess, brenner-scott, FJS-inf-sup, xu-zikatanov-BBtheory, boffi-brezzi-demkowicz-duran-falk-fortin2006}, that under the assumption \eqref{inf-sup_h}, the problem  \eqref{abstract:variational-h} has unique solution $(\bu_h, p_h)\in \bV_h \times \M_h$. 
 
 \begin{theorem}\label{th:T2}  If  $(\bu, p) \in \bV \times Q$  is the solution of  \eqref{abstract:variational}  and $(\bu_h, p_h)\in \bV_h \times \M_h$  is the solution of  \eqref{abstract:variational-h},  then the following inequality holds

 \begin{equation}\label{eq:normu-p-h} 
 |\bu -\bu_h|^2  + \|p-p_h\|^2
\leq   C_h^2    \left ( \du {\inf} {\bv_h \in \bV_h}{} |\bu -\bv_h|^2 + \du{\inf}{q_h \in M_h}{}\|p-q_h\|^2  \right ), 
 \end{equation}
 where 
 \[
 C_h=\frac{\sqrt{4 M^2 +1} +1}{{\sqrt{4m_h^2 +1}-1}}. 
 \]
 \end{theorem}
 \begin{proof}
We consider   the form $\B(\cdot, \cdot)$  defined on  $X=\bV \times Q$  by  
\[
\B((\bu, p), (\bv, q)) := a(\bu, \bv) + b(\bv,p) + b(\bu,q). 
\] 
It is easy to check that  \eqref{abstract:variational} 
is equivalent to the following problem: \\
Find  $(\bu, p) \in \bV \times Q$  such that 
\[
\B((\bu, p), (\bv, q)) = \< F, (\bv, q)\> :=\langle {\bf f},\bv \rangle + \<g,q\>  \  \ \Forall  (\bv, q)  \in \bV \times Q. 
\]
With the  natural inner product defined by \eqref{eq:natiralIP},  the operator $T$  induced by the form $\B$ is exactly 
\[
T: = \begin{pmatrix} 
    I   &    \A^{-1} B^* \\
      \C^{-1} B &  0  
   \end{pmatrix},
\]  
and the corresponding discrete operator  is 
\[
   T_h= \begin{pmatrix} 
    I   &    \A_h^{-1} B_h^* \\
      \C_h^{-1} B_h &  0  
   \end{pmatrix}. 
\]
  
Now, we apply Theorem \ref{th:Xu-Zikatanov} for  the form $\B(\cdot, \cdot)$  defined on  $X=\bV \times Q$. 
 Using the description of $M_\B$ and $m_{\B_h}$ at the end of Section \ref{section:approximation} and the spectral estimate  \eqref{eq:SpT} of Lemma \ref{lemma:1} for both $T$ and $T_h$, we obtain 
 \[
M_\B = \|T\|\  =\max  \{|\lambda| : \lambda \in \sigma(T)\}=  \frac{1}{2} \left (\sqrt{4 M^2 +1} +1\right ),
\]
 and
  \[
m_{\B_h}  = {1}/ { \|T_h^{-1}\|}  =\min  \{|\lambda| : \lambda \in \sigma(T_h)\}= \frac{1}{2}\left (\sqrt{4m_h^2 +1}-1\right ). 
 \]
 The estimate \eqref{eq:normu-p-h} follows as   a direct consequence of  \eqref{eq:Error-h}.
  \end{proof}
 \begin{remark}
 If $m_h \to 0$, then we have that $C_h = O(m_h^{-2})$. The order can be improved if  we equip  $X=\bV \times Q$ or  $X_h=\bV_h \times \M_h$ with the  {\it weighted} on $Q$ inner product   defined by,    
   \[
 \left  (
   \begin{pmatrix}
  \bu  \\
     p   
   \end{pmatrix} , 
   \begin{pmatrix}
  \bv  \\
     q   
   \end{pmatrix}
 \right )_{{m_h}}:=a(\bu, \bv) + m_h^2\,  (p, q).  
   \]
The  operators ${T}$ and ${T}_h$ corresponding to the same form $\B$ and the new weighted inner product  are 
\[
 {T} = \begin{pmatrix} 
    I   &    \A^{-1} B^* \\
     m_h^{-2} \C^{-1} B &  0  
  \end{pmatrix}, \ \ \text{and} \ \ \
   { T_h}= \begin{pmatrix} 
   I   &    \A_h^{-1} B_h^* \\
  m_h^{-2}   \C_h^{-1} B_h &  0  
   \end{pmatrix}. 
\]
Using the same arguments as in the proof of Theorem \ref{th:T2}, we get 
 \begin{equation}\label{eq:normu-p-h2} 
|\bu -\bu_h|^2  + m_h^2 \|p-p_h\|^2
\leq    D_h^2 ( \du {\inf} {\bv_h \in \bV_h}{} |\bu -\bv_h|^2 + m_h^2  \du{\inf}{q_h \in M_h}{}\|p-q_h\|^2),  
 \end{equation}
where 
\[
  D_h={\|T\|} \, \|T_h^{-1}\|= \frac{\sqrt { 4 \frac{M^2} {m_h^2}+1} +1}{2}\,  \frac{2}{\sqrt{5} -1}= O(m_h^{-1}). 
\]
Then, from \eqref{eq:normu-p-h2}, we obtain 
\begin{equation}\label{eq:est1}
|\bu -\bu_h|  \lesssim  \frac{1}{m_h} \   \du {\inf} {\bv_h \in \bV_h}{} |\bu -\bv_h| +  \du{\inf}{q_h \in M_h}{}\|p-q_h\|,
\end{equation}
and 
\begin{equation}\label{eq:est2}
\|p-p_h\|   \lesssim \frac{1}{m_h^2} \   \du {\inf} {\bv_h \in \bV_h}{} |\bu -\bv_h| +   \frac{1}{m_h} \  \du{\inf}{q_h \in M_h}{}\|p-q_h\|. 
\end{equation}
 \end{remark}
By  $A(h) \lesssim  B(h)$,  we understand that  $A(h) \leq c\,  B(h)$, for a constant $c$ independent of $h$. 
The estimates \eqref{eq:est1} and \eqref{eq:est2}  can be useful when the solution $(\bu, p)$ exhibits extra regularity, \cite{B08, BB1, BBX1, BBX2}, and the discretization is done on pairs $(\bV_h, \M_h)$ that are not necessarily  stable, but the $\inf-\sup$ constant $m_h$ can be theoretically or numerically estimated. 


 \section{The proof of Lemma \ref{lemma:1}} \label{section:appendix} 
 
We start this section by reviewing a known functional analysis result that describes the spectrum of a bounded symmetric operator on a Hilbert space. According to one of the referees,  ``the result has actually been known for over a century, and is known as {\it Weyl's criterion}''. It can be found in  \cite{reed-simon}, Theorem VII.12, p.237, and it also follows from Proposition  \ref{prop:p2}.

\begin{prop}\label{prop:3} The spectrum $\sigma(T)$  of a  bounded  symmetric operator $T$ on a Hilbert space $X$    satisfies 
 \[
 \sigma(T) = \sigma_p (T) \cup \sigma_c(T), 
 \] 
where $\sigma_p(T)$ is the {\it point spectrum} of $T$ and consists of all eigenvalues of $T$, and $\sigma_c(T)$  is the {\it continuous spectrum} of $T$ and consists of all $\lambda \in  \sigma(T)$ such that $T-\lambda I $ is an 
one-to-one mapping of  $X$ onto a dense proper subspace of $X$. 
Consequently, (using Proposition  \ref{prop:p2}) for any $T$ is symmetric, we have  that $\lambda \in \sigma(T)$ if and only if  there exists a sequence  $\{x_n\} \subset X$, such that 
\[
\|x_n\| =1, \Forall n   , \ \text{and} \  \  \|(T -\lambda) x_n \| \to 0\  \text {as} \ n\to \infty.  
\]
\end{prop}

Next,  we are ready to present a  proof for Lemma \ref{lemma:1}.

\begin{proof}(Lemma  \ref{lemma:1}) 
First, we will  justify \eqref{eq:SpT_S}. 
Let  $\lambda \in \sigma_p(T_{S_0}) $  be an eigenvalue  and let 
$ \begin{pmatrix}  \bu  \\ p \\ \end{pmatrix} \in \bV \times Q  $  be a corresponding eigenvector. Then,  
\[
T_{S_0}  \begin{pmatrix}  \bu \\ p \\ \end{pmatrix} = \lambda  \begin{pmatrix}  \bu  \\ p \\ \end{pmatrix},
\]
which leads to 
\begin{equation}\label{eq:sis1}
\begin{array}{lccl}
\A^{-1} B^* p  &= & (\lambda-1)   \bu,\\
S_0^{-1} C^{-1} B\bu   & = & \lambda p. 
\end{array}
\end{equation} 
From \eqref{eq:sis1} it is easy to see that for any non-zero  $\bu_0 \in \bV_0=\ker(B)$ we have that 
$  \begin{pmatrix}  \bu_0  \\ 0 \\ \end{pmatrix} $ is an eigenvector for $T_{S_0} $ corresponding to $\lambda =1$. 
Thus $1 \in \sigma_p(T_{S_0})$. If $\lambda \neq 1$, we can assume $p\neq 0$, and by substituting  $\bu$ from  the first equation of  \eqref{eq:sis1}  into the second equation  of  \eqref{eq:sis1}, we get 
\[
S_0^{-1}  S_0 p = \lambda(\lambda -1) p,  \ \text{or} \  \lambda(\lambda -1) =1,
\]
which gives $\lambda =\lambda_{\pm} : = \frac{1 \pm \sqrt{5}}{2}$. Note that for any $p\neq 0$ we have that 
$  \begin{pmatrix}  \lambda_{\pm} \, \A^{-1} B^* p  \\ p \\ \end{pmatrix} $ is an eigenvector for $T_{S_0} $ corresponding to $\lambda =\lambda_{\pm}$. Thus, 
\[
\sigma_p(T_{S_0})=\left \{\frac{1-\sqrt{5}}{2}, 1,  \frac{1+\sqrt{5}}{2} \right \}.
\]
Next, we prove that the continuous spectrum of  $T_{S_0}$ is empty. If we let $\lambda \in  \sigma_c(T_{S_0})$, then $\lambda \neq 1, \lambda \neq  \frac{1\pm \sqrt{5}}{2}$ and, according to Proposition \ref{prop:3},  
 there exists a sequence   
$ \begin{pmatrix}  \bu_n  \\ p_n \\ \end{pmatrix}  \in \bV \times Q $,  such that 

\[
\left \|  \begin{pmatrix}  \bu_n  \\ p_n \\ \end{pmatrix} \right \| =1\ \ \text{ and } \ \   \left \| T_{S_0}  \begin{pmatrix}  \bu_n  \\ p_n \\ \end{pmatrix} -  \lambda  \begin{pmatrix}  \bu_n  \\ p_n \\ \end{pmatrix} \right \|  \to 0, \ \ \text{as} \ \ n\to \infty. 
\]
This leads to 
\begin{equation}\label{eq:sis1n}
\begin{array}{ll}
 \A^{-1} B^* p_n   +( 1 -\lambda) \bu_n & \to 0,\\
 S_0^{-1} \C^{-1} B \bu_n     -\lambda  p_n & \to 0.  
\end{array}
\end{equation}
From \eqref{eq:sis1n}, using that $\C^{-1} B$ and $S_0^{-1}$  are  continuous operators,  we get 
\begin{equation}\label{eq:sis2n}
\begin{array}{lccl}
 (1 -\lambda) \C^{-1} B \bu_n &  +  & S_0\,  p_n   &  \to 0\\
 \C^{-1} B \bu_n   &  - & \lambda S_0  p_n & \to 0. 
\end{array}
\end{equation}
This implies that  $ S_0  p_n  \to 0$, and consequently $ p_n  \to 0$. From the first part of \eqref{eq:sis1n} we can also conclude that $\bu_n \to 0$. The  convergence $(\bu_n, p_n) \to (0, 0)$ contradicts $\left \|  \begin{pmatrix}  \bu_n  \\ p_n \\ \end{pmatrix} \right \| =1$. Thus, $\sigma_c(T_{S_0})=\emptyset $ and the  proof of \eqref{eq:SpT_S} is complete.

To prove \eqref{eq:SpT},  we start by observing that, as in the previous case,  \\ $\lambda=1 \in \sigma_p(T)$. If  $\lambda \neq 1$  is any other spectral value of $\sigma(T)$, then  there exists a sequence   
$ \begin{pmatrix}  \bu_n  \\ p_n \\ \end{pmatrix}  \in \bV \times Q $,  such that 

\[
\left \|  \begin{pmatrix}  \bu_n  \\ p_n \\ \end{pmatrix} \right \| =1\ \ \text{ and } \ \   \left \| T  \begin{pmatrix}  \bu_n  \\ p_n \\ \end{pmatrix} -  \lambda  \begin{pmatrix}  \bu_n  \\ p_n \\ \end{pmatrix} \right \|  \to 0, \ \ \text{as} \ \ n\to \infty. 
\]
The convergence part of the above statement implies  
\begin{equation}\label{eq:sis3n}
\begin{array}{ll}
 \A^{-1} B^* p_n   +( 1 -\lambda) \bu_n & \to 0,\\
 \C^{-1} B \bu_n     -\lambda  p_n & \to 0.   
\end{array}
\end{equation}
From the first equation of \eqref{eq:sis3n} and 
$|\bu_n|^2 + \|p_n\|^2 =1$ we get that 
\[
1 = \frac{1}{(\lambda -1)^2} \|p_n\|_{S_0}^2 + \|p_n\|^2 \leq \left ( \frac{M^2}{(\lambda -1)^2 } +1\right ) \|p_n\|^2.  
\] 
From  \eqref{eq:sis3n} we obtain that  
\[
S_0\,  p_n -\lambda(\lambda -1) p_n \to 0. 
\]
Thus, using the last two statements and Proposition \ref{prop:3} for   characterizing  the spectral values of $S_0$, we  obtain that $\lambda(\lambda -1) \in \sigma(S_0) \subset [m^2, M^2]$, which proves  the right inclusion of  \eqref{eq:SpT}. To complete the proof  of  \eqref{eq:SpT},  we have to show  that $\lambda^{\pm}_m$ and  
$\lambda^{\pm}_M \in  \sigma(T)$.  
From \eqref{eq:sigmaS0}, we have that $m^2 \in \sigma(S_0)$. In light of  Proposition \ref{prop:3},  we can find a sequence $(p_n) \subset Q$ such that 
\[
\|p_n\|=1 \Forall  n, \ \text{and} \ \ S_0\,  p_n - m^2\,  p_n \to 0. 
\]

Then, if we define $\bu_n : = \frac{1}{ \lambda^{\pm}_m} \A^{-1}B^*p_n $,  it is easy to check that 
\[
|\bu_n|^2 + \|p_n\|^2 \geq   \frac{m^2}{(\lambda_m^{\pm} -1)^2 } +1,  \Forall n, 
\]
and
\[
 \left \| T  \begin{pmatrix}  \bu_n  \\ p_n \\ \end{pmatrix} -  \lambda  \begin{pmatrix}  \bu_n  \\ p_n \\ \end{pmatrix} \right \|  \to 0, \ \ \text{as} \ \ n\to \infty. 
\]
This proves that $\lambda^{\pm}_m  \in  \sigma(T)$. The proof of 
$\lambda^{\pm}_M \in  \sigma(T)$ is similar. 
\end{proof}


 \section{Conclusion} \label{section:conclusion} 

  We presented sharp  stability  and approximation estimates  for a general class of  symmetric saddle point  variational systems. The estimates are based on spectral description of the continuous and discrete symmetric operators that represent  the systems. The  spectrum characterization we provided  is an useful tool for analysis at  continuous or  discrete levels, and can be applied sucesfully  to  the convergence analysis of  iterative methods that are aiming directly to the solution of a continuous saddle point problem. Example of such iterative methods include the Uzawa type algorithms of \cite{B09,  B14, BM12, BacShu13, bansch-morin-nochetto}.
   Using a Schur complement norm for the second variable of a saddle point system with a coercive symmetric form $a$ and a mixed form $b$ satisfying an $\inf-\sup$ condition, we established that the ratio between the norm of the data and the norm of the solution lies in  $  [\frac{1}{\varphi}, \varphi]$, where $\varphi$ is the  {\it golden ratio}. 

 {\bf Note:} A slightly different  version of this note was originally submitted to Numerische Mathematik on August 21, 2013. 
 
\bibliography{NSF2014April}

\def\cprime{$'$} \def\ocirc#1{\ifmmode\setbox0=\hbox{$#1$}\dimen0=\ht0
  \advance\dimen0 by1pt\rlap{\hbox to\wd0{\hss\raise\dimen0
  \hbox{\hskip.2em$\scriptscriptstyle\circ$}\hss}}#1\else {\accent"17 #1}\fi}

\end{document}